 \documentclass[reqno,11pt]{amsart}
 \usepackage{amsmath, amsthm, a4, latexsym, amssymb}
\usepackage[unicode]{hyperref}
\usepackage{srcltx}
\usepackage{xcolor}

\usepackage{mathrsfs}

\def\@rmrk#1#2{\refstepcounter
    {#1}\@ifnextchar[{\@yrmrk{#1}{#2}}{\@xrmrk{#1}{#2}}}

\makeatletter\@addtoreset{equation}{section}\makeatother

 \newfont{\bfit}{cmbxti10 scaled 1200}

\renewcommand{\d}{{\rm d}}

 \newcommand{\e}{{\rm e} }

 \newcommand{\eps}{\varepsilon}

 \newcommand{\R}{\mathbb{R}}
 \newcommand{\N}{\mathbb{N}}
 \newcommand{\Z}{\mathbb{Z}}

 \newcommand{\E}{\mathbb{E}}
 \renewcommand{\P}{\mathbb{P}}
 \def\1{{\mathchoice {1\mskip-4mu\mathrm l} 
{1\mskip-4mu\mathrm l}
{1\mskip-4.5mu\mathrm l} {1\mskip-5mu\mathrm l}}}

\newcommand{\ignore}[1]{}
\newcommand{\heap}[2]{\genfrac{}{}{0pt}{}{#1}{#2}}

\newcommand{\ssup}[1] {{\scriptscriptstyle{({#1}})}}

\renewcommand{\subsection}{\secdef \subsct\sbsect}
\newcommand{\subsct}[2][default]{\refstepcounter{subsection}
\vspace{0.15cm}
{\flushleft\bf \arabic{section}.\arabic{subsection}~\bf #1  }
\nopagebreak\nopagebreak}
\newcommand{\sbsect}[1]{\vspace{0.1cm}\noindent
{\bf #1}\vspace{0.1cm}}

{\nopagebreak {\hfill\rule{2mm}{2mm}}\\ }

\newtheorem{theorem}{Theorem}[section]
\newtheorem{lemma}[theorem]{Lemma}
\newtheorem{cor}[theorem]{Corollary}

\newtheoremstyle{thm}{1.5ex}{1.5ex}{\itshape\rmfamily}{}
{\bfseries\rmfamily}{}{2ex}{}

\newtheoremstyle{rem}{1.3ex}{1.3ex}{\rmfamily}{}
{\itshape\rmfamily}{}{1.5ex}{}
\theoremstyle{rem}
\newtheorem{remark}{{\slshape\sffamily Remark}}[]

\refstepcounter{subsubsection}

\def\thebibliography#1{\section*{References}
  \list%
  {\arabic{enumi}.}
    {\settowidth\labelwidth{[#1]}\leftmargin\labelwidth
    \advance\leftmargin\labelsep
    \parsep0pt\itemsep0pt
    \usecounter{enumi}}
    \def\newblock{\hskip .11em plus .33em minus .07em}
    \sloppy                   
    \sfcode`\.=1000\relax}

\begin{document}

\title[Quenched central limit theorem for the SHE]
{\large Quenched central limit theorem for the stochastic heat equation in weak disorder}
\author[Yannic Br\"oker and Chiranjib Mukherjee]{}
\maketitle
\thispagestyle{empty}
\vspace{-0.5cm}

\centerline{\sc Yannic Br\"oker  \footnote{University of Muenster, Einsteinstrasse 62, Muenster 48149, Germany, {\tt yannic.broeker@uni-muenster.de}} and 
Chiranjib Mukherjee \footnote{University of  Muenster, Einsteinstrasse 62, Muenster 48149, Germany, {\tt chiranjib.mukherjee@uni-muenster.de}
}}

\renewcommand{\thefootnote}{}

\vspace{-0.5cm}
\centerline{\textit{University of Muenster}}
\vspace{0.2cm}

\begin{center}
\today

\end{center}

\begin{quote}{\small {\bf Abstract: }
We continue with  the study of the mollified stochastic heat equation in $d\geq 3$ given by $\d u_{\eps,t}=\frac 12\Delta u_{\eps,t}\d t+ \beta \eps^{(d-2)/2} \,u_{\eps,t} \,\d B_{\eps,t}$ with spatially smoothened 
cylindrical Wiener process $B$, whose (renormalized) Feynman-Kac solution describes the partition function of the continuous directed polymer. 
This partition function defines a (quenched) polymer path measure for every realization of the noise and we prove that as long as $\beta>0$ stays small enough, the distribution of the diffusively rescaled Brownian path converges under the aforementioned polymer path measure to the standard Gaussian distribution. }
\end{quote}

\section{Introduction and the result}

\subsection{Motivation.}
We continue the study of the {\it{stochastic heat equation}} (SHE) with multiplicative space-time white noise, formally written as 
\begin{equation}\label{she-formal}
\partial_t u(t,x)=\frac{1}{2}\Delta u(t,x)+u(t,x)\eta(t,x) 
\end{equation}
with $\eta$ being  a centered Gaussian process  with covariance $\mathbf E[\eta(t,x)\,\eta(s,y)]= \delta_0(s-t)\delta_0(x-y)$ for $s,t>0$ and $x,y\in\R^d$.
Note that the Cole-Hopf transformation $h:=-\log u$ translates the SHE to the {\it{Kardar-Parisi-Zhang}} (KPZ) equation, which is a non-linear stochastic partial differential equation also written formally as
$$
\partial_t h(t,x)=\frac{1}{2}\Delta h(t,x) - \frac 12 (\partial_x h(t,x))^2+ \eta(t,x).
$$
Note that both SHE and KPZ are a-priori ill-posed, as only distribution valued solutions are expected for both equations which carry fundamental obstacles arising from multiplying or squaring 
distributions. When the spatial dimension is one,  both equations can be analyzed on a rigorous level, as they
turn out to be the scaling limit of front propagation of some exclusion processes (\cite{BG97},\cite{SS10},\cite{ACQ11}). An intrinsic precise construction of their solutions also yields to
the powerful theories of {\it{regularity structures}} (\cite{H13}) as well as {\it{paracontrolled distributions}} (\cite{GP17}).

In the discrete lattice $\Z^d$, the solution of the SHE is directly related to the {\it{partition function}} of the {\it{discrete directed polymer}}, which is a well-studied model in statistical mechanics (see \cite{AKQ}, \cite{CSY04}).
The directed polymer measure is defined as
\begin{equation}\label{discretepolym}
\mu_n(\d \omega)= \frac 1{Z_n} \exp\bigg\{\beta \sum_{i=1}^n \eta\big(i,\omega(i)\big)\bigg\} \d \mathbb P_0,
\end{equation}
and in this scenario, the space-time white noise potential is replaced by i.i.d. random variables $\eta=\{\eta(n,x)\colon n\in\N, x\in \Z^d\}$ and the strength of the noise is captured by the disorder strength $\beta>0$.
If $\mathbf P$ denotes the law of the potential $\eta$ with $\mathbb P_0$ denoting the distribution of a simple random walk $\omega_n=\big(\omega(i)\big)_{i\leq n}$ 
starting at the origin and independent of the noise $\eta$, and $Z_n=\mathbb E_0[\exp\{\beta \sum_{i=1}^n \eta\big(i,\omega(i)\big)\}]$ denotes the normalizing constant, or the {\it{partition function}} of the discrete directed polymer, 
it is well-known that, when $d\geq 3$, 
the renormalized partition function
$Z_n / \mathbf E[Z_n]$ converges almost surely to a 
random variable $Z_\infty$, which, when $\beta$ is small enough, is positive almost surely (i.e., {\it{weak disorder}} persists \cite{IS88, B89}), and 
in this case, the distribution $\mu_n \, \big(\frac{\omega(n)}{\sqrt n}\big)^{-1}$ of the rescaled paths converges, for any realization of the noise $\eta$, to a centered non-degenerate Gaussian distribution (\cite{CY06}).
On the other hand, for $\beta$ large enough, the limiting random variable satisfies $Z_\infty=0$ (i.e.,
{\it{strong disorder holds}} \cite{CSY04}, \cite{CSY03}).

\subsection{The result.}
We turn to the scenario in the continuum in $d\geq 3$. Note that the equation \eqref{she-formal} can also be written (formally) as an SDE
\begin{equation}\label{stheateqSDE}
\d u_t= \frac 12 \Delta u_t \d t+ \beta\, u_t\, \d B_t,
\end{equation}
where $B_t$ is now a {\it{cylindrical Wiener process}}. That is, the family $\{B_t(f)\}_{f\in \mathcal S(\R^d)}$ is a centered Gaussian process  with covariance given by $\mathbf E[B_t(f)\, B_s(g)]=(t\wedge s) \langle f,g\rangle_{L^2(\R^d)}$ for Schwartz functions $f,g\in \mathcal S(\R^d)$. Defining \eqref{stheateqSDE} precisely requires studying 
a spatially smoothened version 
$$
B_\eps(t,x)= B_t(\phi_\eps(x-\cdot)\,)
$$
of $B_t$ for any smooth mollifier $\phi_\eps(x)=\eps^{-d}\phi(x/\eps)$. Here $\phi$  is chosen to be positive, even, smooth 
and compactly supported and normalized to have total mass $\int \phi=1$. If we write 
\begin{equation}
V_\eps(x)= (\phi_\eps\star \phi_\eps)(x) \qquad V(x)= (\phi\star \phi)(x),
\end{equation}
note that, for any fixed $\eps>0$,  $B_\eps$ is also a centered Gaussian process with covariance kernel $\mathbf E[B_\eps(t,x) \, B_\eps(s,y)]= (t\wedge s) V_\eps(x-y)$. If we denote by $\mathbb P_x$ the law of a Brownian motion $W=(W_t)_t$ started at $x\in \R^d$ and independent of the process $B$, then 
\begin{equation}\label{FK}
u_\eps(t,x)= \mathbb E_x\bigg[\exp\bigg\{\beta\eps^{(d-2)/2}\,\int_0^t\int_{\R^d}\, \phi_\eps(W_s-y) \,\dot B(t-s,\d y)\,\d s-\frac{\beta^2\eps^{d-2}}{2}tV_{\eps}(0)\bigg\}\bigg]
\end{equation}
represents the renormalized Feynman-Kac solution of the mollified stochastic heat equation 
\begin{equation}\label{SHE-mollified}
\begin{aligned}
& \d u_{\eps,t}= \frac 12 \Delta u_{\eps,t}\mathrm d t+ \beta\, \eps^{\frac{d-2}2}\, \, u_{\eps,t} \, \d B_{\eps,t}\qquad d\geq 3,\, \beta>0\\
& u_{\eps,0}= 1.
 \end{aligned}
\end{equation}
Clearly, 
$\mathbf E[u_\eps(t,x)]=1$. By time reversal, for any fixed $t>0$ and $\eps>0$,
\begin{equation}\label{scaling}
 u_{\eps} (t, \cdot)\,\, \overset{\ssup d}= \,\,M_{\eps^{-2}t}(\eps^{-1} \cdot)\;,
\end{equation}
where 
\begin{equation} \label{eq:M}
M_{T}(x) =M_{\beta,T}(x)= \E_x\left[ \exp\left\{\beta \int_0^T \int_{\R^d} \phi(W_s-y) \dot B(s,\d y) \d s\,  - \frac{\beta^2T}{2}V(0)\right\} \right].
\end{equation}
See (Eq.(2.6) in \cite{MSZ16}) for details. Then it was shown \cite[Theorem 2.1 and Remark 2.2]{MSZ16} that for $\beta>0$ sufficiently small and any test function $f\in\mathcal C_c^\infty(\R^d)$, 
\begin{equation}\label{eq-MSZ}
\int_{\R^d} u_\eps(t,x) \, f(x)\,\d x\to \int_{\R^d} \overline u(t,x) \, f(x)\,\d x
\end{equation}
 as $\eps\to 0$ in probability, with $\overline u$ solving the heat equation 
$\partial_t \overline u= \frac 12 \Delta \overline u$ with unperturbed diffusion coefficient. Furthermore, it was also shown in \cite{MSZ16} that, with $\beta$ small enough, and 
for any $t>0$ and $x\in \R^d$,  
$u_{\eps,t}(x) $ converges in law to a non-degenerate random variable $M_\infty$ which is almost surely strictly positive, while $u_{\eps,t}(x)$ converges in probability to zero if $\beta$ is chosen large. 
The pointwise fluctuations of $M_T(x)$ and $u_{\eps,t}(x)$ were studied also in a recent article (\cite{CCM18}) when $\beta$ is sufficiently small (in particular when $M_\infty$ is strictly positive). 
In particular, it was shown that (see \cite[Theorem 1.1 and Corollary 1.2]{CCM18}), in this regime, 
$$
T^{\frac{d-2}4}\bigg(\frac{M_T(x)-M_\infty(x)}{M_T(x)}\bigg)\Rightarrow N\big(0,\sigma^2(\beta)\big) \qquad \mbox{and}\,\,\, \sigma^2(\beta)\neq 1.
$$

Note that in view of the Feynman-Kac representation \eqref{FK}, $u_{\eps}(t,x)$ and $M_T(x)$ are directly related to the (renormalized) {\it{partition function of the continuous directed polymer}}, 
and following the terminology for discrete directed polymer, the strictly positive limit $M_\infty$ for small disorder strength $\beta$ is referred to as the {\it{weak-disorder regime}}, while for $\beta$ large, a vanishing partition function $\lim_{T\to\infty} M_T$ underlines the {\it{strong disorder phase}}. In fact, the polymer model corresponding to \eqref{eq:M} is known as the {\it{Brownian directed polymer in a Gaussian environment}} and the reader is refered to \cite{CC18} for a review of a similar model driven by a Poissonian noise (see also \cite{CY05}, \cite{C18}).

Despite the aforementioned recent results pertaining to the partition function for the continuous directed polymer, the investigation of the actual {\it{polymer path measure}} had remained open. Note that 
the (quenched) polymer path measure  is defined as
\begin{equation}\label{q-polym-meas}
\d\widehat{\mathbb Q}_{\beta,T}^{\ssup x}=\frac 1 {Z_{\beta, T}} \,\, \exp\bigg\{\beta\,\int_0^T\int_{\R^d}\, \phi(W_s-y) \,\dot B(s,\d y)\,\d s\bigg\} \d \P_x,
\end{equation}
for every realization of the  noise $B$. Here $Z_{\beta,T}$ is the un-normalized partition function, i.e., 
$$
\begin{aligned}
Z_{\beta,T}= \E_x\bigg[\exp\bigg\{\beta\,\int_0^T\int_{\R^d}\, \phi(W_s-y) \,\dot B(s,\d y)\,\d s\bigg\}\bigg]&= M_T(x)\, \mathbf E[Z_{\beta,T}]\\
&= M_T(x) \exp\bigg\{\frac{\beta^2 T}2 V(0)\bigg\}.
\end{aligned}
$$
Throughout this article we will assume that $d\geq 3$ and 
$$
\beta<\beta_{L^2}:=\sup\bigg\{b>0\colon\, \sup_T \|M_T\|_{L^2(\mathbf P)}<\infty\bigg\} \in (0,\infty)^\dagger\footnote{$\dagger$Note that a standard Gaussian computation implies that 
$\mathbf E[M_T^2]=\E_0\big[\exp\{\beta^2\int_0^T V(\sqrt 2 W_s)\d s\}\big] \leq \E_0\big[\exp\{\beta^2\int_0^\infty V(\sqrt 2 W_s)\d s\}\big]$. Since $d\geq 3$ and $V\geq 0$ is a continuous function with compact support, 
$\beta^2 \sup_{x\in\R^d} \E_x[\int_0^\infty V(\sqrt 2 W_s)\d s]=\eta < 1$ as soon as $\beta>0$ is chosen small enough. Then by Khas'minski's lemma, $\sup_{x\in\R^d} \E_x[\exp\{\beta^2 \int_0^\infty V(\sqrt 2 W_s)\d s\}]=\frac \eta{1-\eta}<\infty$, and hence $\beta_{L^2} \in (0,\infty)$.}
$$ 
Then the goal of the present article is to show that, for almost every realization of the noise $B$, 
the law of the diffusively rescaled Brownian path under $\widehat{\mathbb Q}_{\beta,T}^{\ssup x}$ converges to the standard Gaussian law. 
We turn to a precise statement of our main result.


\begin{theorem}
	\label{key step}
	Let us assume that $d\geq 3$ and $\beta<\beta_{L^2}$. Then for any $x\in \R^d$ and $\mathbf P$-almost surely, 
	the distribution $\widehat{\mathbb{Q}}_{\beta,T}^{\ssup x}\,\, \big(\frac{ W_{T}} {\sqrt T}\big)^{-1}$ converges weakly to a $d$-dimensional centered Gaussian measure with identity covariance matrix. 
\end{theorem}

\begin{remark}
	For the discrete directed polymer (recall \eqref{discretepolym}), a stronger version of Theorem \ref{key step} was obtained (\cite[Theorem 1.2]{CY06}) for $d\geq 3$ and disorder strength $\beta<\beta_c(d)$ such that the renormalized partition function $Z_n/\mathbf{E} [Z_n]$ converges to a strictly positive random variable (i.e., the {\it{whole weak disorder region}} is considered where $Z_n/\mathbf E[Z_n]$ is not necessarily $L^2(\mathbf P)$-bounded). 
	Also, the result in \cite{CY06} covers not only convergence of the one-time marginal but also convergence of the process, i.e., it is shown that for any suitable test function $F$ on the path space, $\mathbb E^{\mu_n}[F(\omega(n\cdot)/\sqrt n)] \to \mathbb E_0[F(W)]$ in probability with respect to $\mathbf P$ and  $W$ is a Brownian motion in $\R^d$. However, the convergence assertion in \cite{CY06} holds in probability w.r.t $\mathbf P$ unlike the almost sure statement in Theorem \ref{key step}. We believe that using the approach in \cite{CY06}, Theorem \ref{key step} can be extended to the whole weak disorder region (i.e., when $M_\infty$ is a non-degenerate strictly positive random variable). 
	\qed
\end{remark}

Theorem \ref{key step} also implies a central limit theorem for the path measures for the mollified stochastic heat equation \eqref{SHE-mollified}. 
Recall the relation \eqref{scaling}, and note that, for any fixed $t>0$, 
\begin{equation}\label{q-polym-meas}
\d\widetilde{\mathbb Q}_{\beta,\eps,t}=\frac 1 {Z_{\beta,\eps,t}} \,\, \exp\bigg\{\beta\eps^{\frac{d-2}2}\,\int_0^t\int_{\R^d}\, \phi_\eps(W_s-y) \,\dot B(t-s,\d y)\,\d s\bigg\} \d \P_0,
\end{equation}
defines the (quenched) path measure for \eqref{SHE-mollified}, and here 
$$
Z_{\beta,\eps,t}=u_{\eps,t}(0) \, \exp\bigg\{\frac {t \beta^2}2 \eps^{(d-2)/2} \, \, V_\eps(0)\bigg\}.
$$
The following result then will be a direct consequence of Theorem \ref{key step}.

\begin{cor}
	\label{CLT for SHE-1}
	Let $d\geq 3$ and $\beta<\beta_{L^2}$ as in Theorem \ref{key step}. Then for any fixed $t>0$, 
	$$
	\widetilde{\mathbb Q}_{\beta,\eps,t} \,\,  (\eps W_{\eps^{-2}})^{-1}\,\, {\Rightarrow} \,\, \mathbf{N}(\mathbf{0},\mathbf{I_{d}})
	$$ 
	and the above convergence holds in probability with respect to the law $\mathbf P$ of the noise $B$.
\end{cor}

\begin{proof}[{\bf{Proof of Corollary \ref{CLT for SHE-1} (assuming Theorem \ref{key step})}}]
	Since  $t>0$ is fixed, for simplicity we will take $t=1$ and prove the result for $\widetilde{\mathbb Q}_{\beta,\eps}= \widetilde{\mathbb Q}_{\beta,\eps,1}$. 
	It suffices to show that 
	$$
	\E^{\widetilde{\mathbb Q}_{\beta,\eps}}\bigg[\e^{\eps\langle\lambda, W_{\eps^{-2}}\rangle}\bigg] \to \e^{|\lambda|^2/2} \qquad\mbox{in probability w.r.t}\,\,\,\mathbf P.
	$$
	Then  for $\eps<1$, 
	\begin{align*}
		\mathbb E_0&\bigg[\exp\Big\{\beta\eps^{(d-2)/2}\int_0^1\int_{\R^d}\phi_{\eps}(W_s-y)\dot B(1-s,\d y)\d s\Big\}\e^{\eps\langle\lambda,W_{\eps^{-2}}\rangle}\bigg]\\
		&=\mathbb E_0\bigg[\exp\Big\{\beta\eps^{(d-2)/2}\int_0^1\int_{\R^d}\phi_{\eps}(W_s-y)\dot B(1-s,\d y)\d s\Big\}\e^{\eps\langle\lambda,W_1\rangle}\e^{\eps\langle\lambda,W_{\eps^{-2}}-W_1\rangle}\bigg]\\
		&=\mathbb E_0\bigg[\exp\Big\{\beta\eps^{(d-2)/2}\int_0^1\int_{\R^d}\phi_{\eps}(W_s-y)\dot B(1-s,\d y)\d s\Big\}\e^{\eps\langle\lambda,W_1\rangle}\bigg]\mathbb E_0\Big[\e^{\eps\langle\lambda,W_{\eps^{-2}}-W_1\rangle}\Big]\\
		&=\mathbb E_0\bigg[\exp\Big\{\beta\eps^{(d-2)/2}\int_0^1\int_{\R^d}\phi_{\eps}(W_s-y)\dot B(1-s,\d y)\d s\Big\}\e^{\eps\langle\lambda,W_1\rangle}\bigg]\e^{\frac{|\lambda|^2}{2}(1-\eps^2)}, 
	\end{align*}
         and we have used Markov property for $W$ at time $1$. If we now recall \eqref{scaling}, 
	\begin{align*}
	\mathbb E_0\bigg[&\exp\Big\{\beta\eps^{(d-2)/2}\int_0^1\int_{\R^d}\phi_{\eps}(W_s-y)\dot B(1-s,\d y)\d s\Big\}\e^{\eps\langle\lambda,W_1\rangle}\bigg]\\
	&\overset{(d)}{=}\mathbb E_0\bigg[\exp\Big\{\beta\int_0^{\eps^{-2}}\int_{\R^d}\phi(y-W_s)\dot B(s,\d y)\d s\Big\}\e^{\eps^2\langle\lambda,W_{\eps^{-2}}\rangle}\bigg].
	\end{align*}
	From now on we will  abbreviate
	$$
	H_{\beta,\eps}(W,B)= \beta\int_0^{\eps^{-2}}\int_{\R^d}\phi(y-W_s)\dot B(s,\d y)\d s-\frac{\beta^2}{2\eps^2}V(0).
	$$
	Then, we have
	$$
	\begin{aligned}
	\E^{\widetilde{\mathbb Q}_{\beta,\eps}}\big[\e^{\eps\langle \lambda, W_{\eps^{-2}}\rangle}\big] 
	&\overset{(d)}{=}\e^{\frac{|\lambda|^2}{2}(1-\eps^2)}\,\,\frac{\mathbb E_0\big[\exp\big\{H_{\beta,\eps}(W,B)+{\eps\langle\lambda,\,\,\eps \,\,W_{\eps^{-2}}\rangle}\big\}\big]}{\mathbb E_0\big[\exp\big\{H_{\beta,\eps}(W,B)\big\}\big]} \\
	&=\e^{\frac{|\lambda|^2}{2}\big(1-\frac 1 T\big)}\,\,\E^{\widehat{\mathbb Q}_{\beta,T}}\big[\e^{\frac 1 {\sqrt T}\langle \lambda, \,\,\frac{W_{T}}{\sqrt T}\rangle}\big] \\
	&\to \e^{|\lambda|^2/2},
	\end{aligned}
	$$
	where, in the second identity above, we wrote $T:=\eps^{-2}$ and the last statement holds true for $\mathbf{P}$-almost every realization of $B$ and follows from Theorem \ref{key step}. 
	Indeed, we can expand the exponential $\exp\{\frac 1 {\sqrt T}\langle \lambda, \,\,\frac{W_{T}}{\sqrt T}\rangle\}$
	into a power series, and since all moments $\E^{\widehat{\mathbb Q}_{\beta,T}}[\langle \lambda, \,\,\frac{W_T}{\sqrt T}\rangle^n]$ converge, according to Theorem \ref{key step}, to the moments $E[\langle \lambda, X\rangle^n]$ of a Gaussian $X\sim N(0, \mathbf{I_{d}})$, the expectation $\E^{\widehat{\mathbb Q}_{\beta,T}}\big[\exp\big\{\frac 1 {\sqrt T}\langle \lambda, \,\,\frac{W_{T}}{\sqrt T}\rangle\big\}\big]$ converges $\mathbf P$-a.s. to $1$.	This concludes the proof of Corollary \ref{CLT for SHE-1}.
\end{proof}

\begin{remark}\label{remark2}
	As remarked earlier, in \cite{MSZ16} it was shown that the solution $u_{\eps,t}(x)$ of the stochastic heat equation \eqref{SHE-mollified} with constant initial condition $1$ converges in probability w.r.t. $\mathbf P$ in the weak disorder phase (i.e., for $\beta>0$ small enough) to a strictly positive random variable $M_\infty(x)$. Note that the argument for the proof of Corollary \ref{CLT for SHE-1} also implies the convergence of the ratio
	$$
	\frac{ \widehat u_\eps^{\ssup\lambda}}{u_\eps} \to 1 \qquad\text{as }\eps\to 0,
	$$
	where $\widehat u_\eps^{\ssup\lambda}$ denotes the solution of the stochastic heat equation \eqref{SHE-mollified} with initial condition $u_{\eps,0}^{\ssup\lambda}(x)=f(x)=\e^{\eps\langle\lambda, x\rangle}$.
	Although a proof of Corollary \ref{CLT for SHE-1} could probably be given by first proving the convergence of the above ratio following analytical tools, for our purposes we choose to rely on more probabilistic arguments based on $L^2(\mathbf{P})$ computations and martingale methods as in \cite{MSZ16} and \cite{B89}.\qed
\end{remark}

\begin{remark}\label{remark1}
	The case when the noise $B$ is smoothened both in time and space has also recently been considered. If $F(t,x)= \int_{\R^d}\int_0^\infty \phi_1(t-s) \phi_2(x-y) \d B(s,y)$ is the mollified noise, and $\hat u_\eps(t,x)=u(\eps^{-2}t,\eps^{-1}x)$ with $u$ solving $\partial_t u=\frac 12 \Delta u + \beta \, F(t,x) u$ with initial condition $u(0,x)=u_0(x)$ with $u_0\in C_b(\R^d)$, then it was shown in \cite{M17} that for any $\beta>0$ and $x\in \R^d$, $\frac 1 {\kappa(\eps,t)}\,\, \mathbb E[\hat u_\eps(t,x)] \to \hat u(t,x)$ as $\eps\to 0$, where $\kappa(\eps,t)$ is a divergent constant and $\hat u(t,x)$ solves the homogenized heat equation 
\begin{equation}\label{eq-heat-homog}
\partial_t \hat u= \frac 12 \mathrm{div}\big(\mathrm{a_\beta} \nabla \hat u\big)
\end{equation}
with diffusion coefficient $\mathrm a_\beta\neq \mathbf {I}_d$. It was shown in \cite{MU17,GRZ17} 
that, for $\beta>0$ small enough, the rescaled and spatially averaged fluctuations converge, i.e., 
$$
\eps^{1-d/2}\int \d x f(x)[\hat u_\eps(t,x)-\mathbf E(\hat u_\eps(t,x)] \Rightarrow \int f(x) \mathscr U(t,x) \, \d x 
$$ 
and the limit $\mathscr U$ satisfies the additive noise stochastic heat equation 
$$
\partial_t \mathscr U=  \frac 12 \mathrm{div}\big(\mathrm{a_\beta} \nabla \mathscr U\big) + \beta \nu^2(\beta) \,\hat u \, \d B
$$
 with diffusivity $\mathrm a_\beta$ and variance $\nu^2(\beta)$, and $\hat u$ solves \eqref{eq-heat-homog}. 
\end{remark}

\begin{remark}
Finally we remark on the $T\to\infty$-asymptotic behavior of the polymer measures $\widehat{\mathbb Q}_{\beta,T}$ when $\beta$ is large. Recall that, for large $\beta$, it was shown in \cite{MSZ16} that the renormalized partition function $M_T$ converges in distribution to $0$ (i.e., we are in the strong disorder regime). In a recent article (\cite{BM18}), based on the compactness theory developed in \cite{MV14}, we show that for large enough $\beta$, the distribution of the endpoint of the path $W_T$ under $\widehat{\mathbb Q}_{\beta,T}$ is concentrated in random spatial regions, leading to a strong localization effect.
\end{remark}

The rest of the article is devoted to the proof of Theorem \ref{key step}.

\section{Proof of Theorem \ref{key step}}

We remind the reader that $V=\phi\star\phi$ where $\phi$ is a smooth, positive, even function and has support in a ball of radius $K$ around the origin. Moreover, $\int_{\R^d} \phi=1$. 
Furthermore, $\P_0$ denotes the law of a Brownian motion $W$ in $\R^d$, starting from $0$, with $\E_0$ denoting the corresponding expectation, while $\mathbf P$ denotes the law of the white noise $B$ which is independent of $W$ and $\mathbf E$ denotes the corresponding expectation. We will also denote by $\mathcal{F}_T$ the $\sigma$-field generated by the noise $B$ up to time $T$.

For simplicity, we will fix $x=0$ and we will show that, for $\beta<\beta_{L^2}$ and any $\lambda\in \R^d$, $\mathbf P$-almost surely, 
\begin{equation}\label{claim}
\lim_{T\to\infty} \E^{\widehat{\mathbb Q}_{\beta,T}} \bigg[ \e^{\langle \lambda , W_T/\sqrt T\rangle} \bigg] = \e^{|\lambda|^2/2}
\end{equation}
where $\widehat{\mathbb Q}_{\beta,T}=\widehat{\mathbb Q}_{\beta,T}^{\ssup 0}$.





\begin{lemma}
	\label{lem2}
	In $d\geq 3$ and for $\beta>0$ small enough, $M_T$ converges almost surely to a random variable $M_\infty$ that satisfies
	$$\mathbf E[M_{\infty}]=1\quad\text{and}\quad\mathbf P(M_{\infty}=0)=0.$$
\end{lemma}
\begin{proof}
Let us set
\begin{equation}\label{eq-H-def}
H_{\beta,T}(W,B)=\beta\int_0^T\int_{\mathbb{R}^d}\phi(y-W_s)\dot B(s,\d y)\d s-\frac{\beta^2T}{2}V(0)
\end{equation}\
Then the proof follows from the fact that
$M_T= \exp\{H_{\beta,T}\}$
is a nonnegative $(\mathcal{F}_T)$-martingale, which remains bounded in $L^2(\mathbf P)$ for $\beta<\beta_{L^2}$.
\end{proof}

\begin{lemma}
	Let $\mathcal G_T$ be the $\sigma$-algebra generated by the Brownian path $(W_s)_{0\leq s\leq T}$ until time $T$. Then, for $n=(n_1,\dots,n_d)\in \N_0^d$ with $|n|:=\sum_{i=1}^d n_i$, and $\lambda=(\lambda_1,\dots,\lambda_d)\in\R^d$, 
	\begin{equation}\label{def-In}
	I_n(T,W_T)= \frac {\partial^{|n|}}{\prod_{j=1}^d \partial\lambda_j^{n_j}}  \, \bigg[\exp\Big\{ \sum_{i=1}^d \lambda_i W^{\ssup i}_T  - \frac 12 |\lambda|^2 T\Big\}\bigg] \, \Bigg|_{\lambda=0}
	\end{equation}
	 is a $(\mathcal{G}_T)$-martingale.
\end{lemma}
\begin{proof}
	We first note that 
	$$
	U_\lambda(T)=\exp\bigg\{\sum_{i=1}^d\lambda_iW_T^{\ssup i}-\frac{1}{2}|\lambda|^2T\bigg\}
	$$
	is a $(\mathcal{G}_T)$-martingale for every $\lambda\in \R^d$. 
	Now if we write $U_\lambda(T)$ in its Taylor series at $\lambda=0$ as
	$$
		U_\lambda(T)=\sum_{n_1=0}^{\infty}\cdot\cdot\cdot\sum_{n_d=0}^{\infty}I_n(T,W_T)\frac{\prod_{i=1}^{d}\lambda_i^{n_i}}{\prod_{i=1}^{d}n_i!},
	$$
	it follows that $I_n(T,W_T)$ is a $(\mathcal{G}_T)$-martingale for every $n\in\N_0^d$.
\end{proof}

\begin{lemma}\label{lemma-Y-mart}
The sequence
$$
Y_n(T)=\mathbb E_0\bigg[\exp\Big\{H_{\beta,T}(W,B)\Big\}I_n(T,W_T)\bigg]
$$
is a martingale with respect to the filtration $\mathcal F_T$ generated by $(B_s)_{0\leq s\leq T}$.
\end{lemma}
\begin{proof}
Since $I_n(T,W_T)$ is a $\mathcal G_T=\sigma(W_s:0\leq s\leq T)$-martingale, we get for $S\leq T$
\begin{align*}
	\mathbf E[Y_n(T)|\mathcal{F}_S]&=\mathbb E_0\Big[\mathbf E\big[\exp\big\{H_{\beta,T}(W,B)\big\}\big|\mathcal{F}_S\big]I_n(T,W_T)\Big]\\
	&=\mathbb E_0\Big[\exp\big\{H_{\beta,S}(W,B)\big\}\mathbb E_0\big[I_n(T,W_T)\big|\mathcal{G}_S\big]\Big]\\
	&=Y_n(S).
\end{align*}
Furthermore, since $I_n(T,W_T)$ is integrable w.r.t $\P_0$, $\mathbf E[\exp\{H_{\beta,T}(W,B)\}]=1$ proves the lemma as 
$\mathbf E[Y_n(T)]
	=\mathbb E_0[I_n(T,W_T)]
	<\infty$. 
\end{proof}
We will now control the martingale difference sequence $Y_n(T)-Y_n(T-1)$ in $L^2(\mathbf P)$ which will provide some estimate on the decay on correlations. 
\begin{lemma}\label{lemma-mart-difference}
If $d\geq 3$ and $\beta<\beta_{L^2}$ as in Theorem \ref{key step}, then there exists $C=C(\beta,d,\phi)$ such that for any $T\geq 1$ and $|n|\geq 1$,
$$
\mathbf E\big[(Y_n(T)-Y_n(T-1))^2\big] \leq C T^{|n|-p}.
$$
for $p=9/8$.
\end{lemma}
\begin{proof}
We compute the martingale difference as follows:
	\begin{align*}
		Y_n&(T)-Y_n(T-1)\\
		&=\mathbb E_0\bigg[\exp\big\{H_{\beta,T}(W,B)\big\}I_n(T,W_T)
		 -\exp\big\{H_{\beta,T-1}(W,B)\big\}I_n(T,W_T)\\
		&\qquad\qquad +\exp\big\{H_{\beta,{T-1}}(W,B)\big\}\Big(I_n(T,W_T)-I_n(T-1,W_{T-1})\Big)\bigg]\\
		&=\mathbb E_0\bigg[\exp\big\{H_{\beta,T-1}(W,B)\big\}\,
		\Big(\exp\big\{H_{\beta,T-1,T}(W,B)\big\}-1\Big)\,I_n(T,W_T)\bigg],
	\end{align*}
	where $H_{\beta,T-1,T}(W,B)$ is nothing but $H_{\beta,T}(W,B)$ with the integral starting in $T-1$ instead of $0$ and we have used that $I_n(T,W_T)$ is a martingale. Then 
	\begin{align}
		&\!\!\!\!\!\!\!\mathbf E\Big[(Y_n(T)-Y_n(T-1))^2\Big]\\
		\begin{split}
		\label{use Hölder}
		&\!\!\!\!\!\!\!=(\mathbb E_0\otimes \mathbb E_0) \bigg[\mathbf E\Big[\exp\Big\{\beta\int_0^{T-1}\int_{\R^d}\big(\phi(y-W_s)+\phi(y-W_s^{\prime})\big)\dot B(s,\d y)\d s-\beta^2(T-1)V(0)\Big\}\Big]\\
		&\!\!\!\!\!\!\!\times\mathbf E\Big[\exp\Big\{\beta\int_{T-1}^T\int_{\R^d}\big(\phi(y-W_s)+\phi(y-W_s^{\prime})\big)\dot B(s,\d y)\d s-\beta^2V(0)\Big\}-1\Big]I_n(T,W_T)I_n(T,W_T^{\prime})\bigg]
	\end{split}
	\end{align}
	with $W$ and $W^\prime$ being two independent copies of the Brownian path. If $K<\infty$ denotes the radius of the support of $\phi$, and if we assume
	 $|W_s-W_s^{\prime}|>2K$, either $\phi(\cdot-W_s)$ or $\phi(\cdot-W_s^{\prime})$ must vanish, everywhere in $\R^d$. Therefore, on the event $\big\{|W_s-W_s^{\prime}|>2K	\,\,\forall\,\, s\in[T-1,T]\big\}$, 
	 $$
	 \begin{aligned}
	 \mathbf E\Big[\exp\Big\{\beta\int_{T-1}^T\int_{\R^d}\big(\phi(y-W_s)+\phi(y-W_s^{\prime})\big)\dot B(s,\d y)\d s\Big\}\Big] &= \bigg[\mathbf E\Big[\exp\Big\{\beta\int_{T-1}^T\int_{\R^d}\phi(y-W_s)\dot B(s,\d y)\d s\Big\}\Big] \bigg]^2 \\
	 &= \e^{\beta^2 V(0)}.
	 \end{aligned}
	 $$
	 Hence, we can estimate
	\begin{align*}
		(\mathbb E_0&\otimes \mathbb E_0) \bigg[\mathbf E\Big[\exp\Big\{\beta\int_{T-1}^T\int_{\R^d}\big(\phi(y-W_s)+\phi(y-W_s^{\prime})\big)\dot B(s,\d y)\d s-\beta^2V(0)\Big\}-1\Big]\bigg]\\
		&=(\mathbb E_0\otimes \mathbb E_0)\bigg[\1_{\{|W_s-W_s^{\prime}|\leq 2K\text{ for some } s\in[T-1,T]\}}\\
		&\times\mathbf E\Big[\exp\Big\{\beta\int_{T-1}^T\int_{\R^d}\big(\phi(y-W_s)+\phi(y-W_s^{\prime})\big)\dot B(s,\d y)\d s-\beta^2V(0)\Big\}-1\Big]\bigg]\\
		&=(\mathbb E_0\otimes \mathbb E_0)\bigg[\1_{\{|W_s-W_s^{\prime}|\leq 2K\text{ for some } s\in[T-1,T]\}}\Big(\exp\Big\{\beta^2\int_{T-1}^TV(W_s-W_s^{\prime})\d s\Big\}-1\Big)\bigg]\\
		&\leq \Big(\e^{\beta^2\|V\|_\infty}-1\Big)(\P_0\otimes \P_0)\big(|W_s-W_s^{\prime}|\leq 2K\text{ for some } s\in[T-1,T]\big),
	\end{align*}
	where the inequality holds as $V=\phi\star \phi$ is bounded.
	
	Since
	\begin{align*}
		(\mathbb E_0&\otimes \mathbb E_0) \bigg[\mathbf E\Big[\exp\Big\{\beta\int_0^{T-1}\int_{\R^d}\big(\phi(y-W_s)+\phi(y-W_s^{\prime})\big)\dot B(s,\d y)\d s-\beta^2(T-1)V(0)\Big\}\Big]\bigg]\\
		&=(\mathbb E_0\otimes \mathbb E_0)\bigg[\exp\Big\{\beta^2\int_0^{T-1}V(W_s-W_s^{\prime})\d s\Big\}\bigg]
	\end{align*}
	and $W_s-W_s^{\prime}\overset{(d)}{=}W_{2s}$, if we now invoke H\"older's inequality with $q=4/3$ and $p=r=8$ to \eqref{use Hölder},  we get
	\begin{align}
		\label{first factor}
		\mathbf E\Big[(Y_n(T)-Y_n(T-1))^2\Big]&\leq\bigg(\mathbb E_0\Big[\exp\Big\{8\beta^2\int_0^{\infty}V(W_{2s})\d s\Big\}\Big]\bigg)^{1/8}\\
		\label{second factor}
		&\times \Big(c(\P_0\otimes \P_0)\big(|W_s-W_s^{\prime}|\leq 2K\text{ for some } s\in[T-1,T]\big)\Big)^{3/4}\\
		\label{third factor}
		&\times\Big((\mathbb E_0\otimes \mathbb E_0)\Big[\big(I_n(T,W_T)I_n(T,W_T^{\prime})\big)^8\Big]\Big)^{1/8}.
	\end{align}
	As $\beta>0$ is chosen small enough and $d\geq 3$, the factor on the right hand side of (\ref{first factor}) is finite (see Lemma 3.1 in \cite{MSZ16}), while for the factor in (\ref{second factor}) we use
	\begin{align}
	(\P_0\otimes \P_0)&\big(|W_s-W_s^{\prime}|\leq 2K\text{ for some } s\in[T-1,T]\big)\nonumber \\
	&\leq\int_{T-1}^T \P_0(W_{2s}\in B_0(2K))\leq c^{\prime}(T-(T-1))(T-1)^{-d/2} \label{fourth factor}	
	\end{align}
	for $T$ large enough and a proper constant $c^{\prime}$. For the factor in (\ref{third factor}), we need some facts regarding the polynomial 
	$$
	I_n(T,x)= \frac {\partial^{|n|}}{\prod_{j=1}^d \partial\lambda_j^{n_j}}  \, \bigg[\exp\Big\{ \sum_{i=1}^d \lambda_i x^{\ssup i}  - \frac 12 |\lambda|^2 T\Big\}\bigg] \, \Bigg|_{\lambda=0}
	$$
	for $x=(x^{\ssup 1},\dots,x^{\ssup d}),\, \lambda=(\lambda_1,\dots,\lambda_d)\in \R^d$ and it is useful to collect them now:
	\begin{lemma}
	\label{lem3}
	Let $T>0$, the polynomial $I_n(T,X_T)$ can be rewritten as
	\begin{equation}\label{star}
	I_n(T,X_T)=\sum_{i_1,...,i_d,j}A_n^{X_T}(i_1,...,i_d,j)X_T^{\ssup 1^{i_1}}\cdot\cdot\cdot X_T^{\ssup d^{i_d}}T^j,
	\end{equation}
	where the coefficients $A_n^{X_T}(i_1,...,i_d,j)$ satisfy the properties
	\begin{itemize}
		\item[(a)] $A_n^{X_T}(i_1,...,i_d,j)=0$ for $i_1+...+i_d+2j\neq |n|$
		\item[(b)] $A_n^{X_T}(i_1,...,i_d,j)=A_n^{X_1}(i_1,...,i_d,j)$
		\item[(c)] $A_n^{X_T}(i_1,...,i_d,0)=\delta_{i_1 n_1}\cdot\cdot\cdot\delta_{i_d n_d}$ for $i_1+...+i_d=|n|.$
	\end{itemize}
\end{lemma}
	We assume Lemma \ref{lem3} for now and continue with the proof of Lemma \ref{lemma-mart-difference}. Note that we only have to estimate  the factor in (\ref{third factor}), for which we can now apply lemma \ref{lem3}(a) and use that $W_T$ and $W_T^{\prime}$ are independent such that
	\begin{align*}
		\Big((\mathbb E_0\otimes \mathbb E_0)\Big[(I_n(T,W_T)I_n(T,W_T^{\prime}))^8\Big]\Big)^{1/8}=\Big(\mathbb E_0\Big[I_n(T,W_T)^8\Big]\Big)^{1/4}.
	\end{align*}
	We claim that  $E_0[I_n(T,W_T)^8]=O(T^{4 |n|})$. Indeed, note that by Lemma \ref{lem3}(a) and \eqref{star}, 
	$$
	I_n(T,W_T)=\sum_{i_1+\dots +i_d+2j=|n|} A_n^{W_T}(i_1,...,i_d,j)W_T^{\ssup 1^{i_1}}\cdot\cdot\cdot W_T^{\ssup d^{i_d}}T^j.
	$$	
	Now we apply multinomial theorem for $I_n(T,W_T)^8$ in the above display whence the expansion yields 
	$$
	I_n(T,W_T)^8 =\sum_{\heap{i_1+\dots +i_d+2j=|n|}{\sum_{\ell}r_\ell=8}} \frac{8!}{\prod_\ell r_\ell!} \prod_\ell \bigg(A_n^{W_T}(i_1,...,i_d,j)W_T^{\ssup 1^{i_1}}\cdot\cdot\cdot W_T^{\ssup d^{i_d}}T^j\bigg)^{r_\ell}.
	$$	
	Since $W_T^{\ssup 1},\dots, W_T^{\ssup d}$ are independent and $\E_0\big[\big(W_T^{\ssup i}\big)^k\big]= T^{k/2} C_k$ for any $k\in \N$ and $i=1,\dots,d$, and furthermore by Lemma \ref{lem3}(b), 
	the coefficients $A_n^{W_T}(i_1,...,i_d,j)$ do not depend on $W_T$ or $T$, we have 
	$$
	\E_0[I_n(T,W_T)^8] =\sum_{\heap{i_1+\dots +i_d+2j=|n|}{\sum_{\ell}r_\ell=8}} C(r_\ell,i_1,\dots,i_d,j) T^{\big(\frac{i_1}2+\dots+\frac{i_d}2+j\big)\sum_\ell r_\ell} = O\big(T^{4|n|}\big),
	$$
	proving that 
	\begin{align}\label{fifth factor}
		\Big((\mathbb E_0\otimes \mathbb E_0)\Big[(I_n(T,W_T)I_n(T,W_T^{\prime}))^8\Big]\Big)^{1/8}=O(T^{|n|}).
	\end{align}		
	 Then \eqref{fourth factor} and \eqref{fifth factor}, together with \eqref{first factor}-\eqref{third factor} imply that for  $d\geq 3$ 
	$$\mathbf E\Big[(Y_n(T)-Y_n(T-1))^2\Big]=O\big(T^{|n|-p}\big).
	$$
	and $p=9/8$. We now owe the reader only the proof of Lemma \ref{lem3}.
	
	\noindent {\bf{Proof of Lemma \ref{lem3}:}} We make two simple observations: 
	\begin{itemize}
		\item[(1)] 
		$$
		\frac{\partial}{\partial \lambda_k} \, \bigg[\exp\Big\{\sum_{i=1}^d\lambda_iX_T^{\ssup i}-\frac{1}{2}|\lambda|^2T\Big\}\bigg] =(X_T^{\ssup k}-\lambda_kT)\exp\Big\{\sum_{i=1}^d\lambda_iX_T^{\ssup i}-\frac{1}{2}|\lambda|^2T\Big\},
		$$
		\item[(2)] 
		$$
		\frac{\partial}{\partial \lambda_k} \, \bigg[\big(X_T^{\ssup k}-\lambda_kT\big)^{i_k}\bigg] =- i_kT\big(X_T^{\ssup k}-\lambda_kT\big)^{i_k-1}.
		$$
	\end{itemize}
	We will now prove Lemma \ref{lem3} by  induction as follows:\\
	For $|n|=0$ we have $I_n(T,X_T)=1$ and thus $i_1+...+i_d+2j=0=|n|$. For the induction step we assume $i_1+...+i_d+2j=|n|$ for every summand in $I_n(T,X_T)$ which is non-zero, where $n\in\N_0^d$ is fixed. If we 
	differentiate  $I_n(T,X_T)$ w.r.t. $\lambda_k$, every summand of the new polynomial changes as written in (1) or (2). Without loss of generality we assume $k=1$. In case (1), only the exponent of $X_T^{\ssup 1}$ increases by $1$ and since the assumption yields $(i_1+1)+i_2+...+i_d+2j=|n|+1$, no summand influences the induction step.\\
	For a summand that follows case (2) the exponent of $T$ increases by $1$, while the exponent of $X_T^{\ssup 1}$ decreases by $1$. The assumption yields $(i_1-1)+i_2+...+i_d+2(j+1)=|n|+1$. From both cases together one can conclude Lemma \ref{lem3} (a).\\
	Furthermore, as the coefficients do not depend on the values of $X_T$ or $T$, part (b) follows immediately.\\
	For the statement in (c) we again analyze the cases (1) and (2). If we again assume $k$ to be one, in case (2) the exponent of $X_T^{\ssup 1}$ decreases by $1$ and in case (1) only $i_1$ increases  by $1$. This means that 
	the requirement $i_1+...+i_d=|n|$ can only be fulfilled by that summand that always follows case (1). Since $A_0(0,...,0)=1$, the statement in Lemma \ref{lem3} (c) now also follows.

This concludes the proof of Lemma \ref{lem3}. Hence Lemma \ref{lemma-mart-difference} is also proven.

	\end{proof}
	\begin{lemma}\label{lem3.5}
	Under the assumptions imposed in Lemma \ref{lemma-mart-difference}, the process $(M_{\tau+N,n})_{N\in\N}$ is an $(\mathcal{F}_{\tau+N})_{N\in\N}$-martingale bounded in $L^2(\mathbf P)$, where
	$$M_{\tau+N,n}=\sum_{S=1}^{N}S^{-|n|/2}(Y_n(\tau+S)-Y_n(\tau+S-1))$$
	and $\tau=T-\lfloor T\rfloor$.
	\end{lemma}
	\begin{proof}
	By Lemma \ref{lemma-Y-mart}, $(Y_n(T))_{T\geq 0}$ is an $(\mathcal{F}_T)_{T\geq 0}$-martingale, and we know that the process $(M_{\tau+N,n})_{N\in\N}$ is an $(\mathcal{F}_{\tau+N})_{N\in\N}$-martingale. Thus
	\begin{align*}
	\mathbf E\big[M_{\tau+N,n}^2\big]&=\sum_{S=1}^{N}S^{-|n|}\mathbf E\big[(Y_n(\tau+S)-Y_n(\tau+S-1))^2\big]\\
	&+\sum_{S=1}^{N}\sum_{S\neq R=1}^{N}S^{-|n|/2}R^{-|n|/2}\mathbf E\big[(Y_n(\tau+S)-Y_n(\tau+S-1))(Y_n(\tau+R)-Y_n(\tau+R-1))\big]\\
	&=\sum_{S=1}^{N}S^{-|n|}\mathbf E\big[(Y_n(\tau+S)-Y_n(\tau+S-1))^2\big]
	\end{align*}
	and by Lemma \ref{lemma-mart-difference}
	$$\mathbf E\Big[M_{\tau+N,n}^2\Big]\leq C\sum_{S=1}^{N}S^{-9/8}\Big(\frac{S+\tau}{S}\Big)^{|n|}\leq 2^{|n|}C\sum_{S=1}^N S^{-9/8}\leq2^{|n|}C\sum_{S=1}^{\infty}S^{-9/8}<\infty$$
	for a proper constant $C$ and $T=\tau+N$ large enough. From this we can conclude that $M_{\tau+N,n}$ is an $L^2$-bounded martingale for every $n\in\N^d\setminus\{0\}$.
	\end{proof}

	\begin{lemma}
	\label{lem4}
	Under the assumptions imposed in Lemma \ref{lemma-mart-difference}, if $n\in\N^d\setminus\{0\}$, then
	$$\lim_{T\rightarrow\infty}T^{-|n|/2}Y_n(T)=0$$
	holds $\mathbf P$-almost surely.
\end{lemma}
     \begin{proof}
    From Lemma \ref{lem3.5} we deduce that $M_{\tau+N,n}$   converges almost surely w.r.t. $\mathbf P$. Then
	$$\lim_{N\rightarrow\infty}N^{-|n|/2}\sum_{S=1}^{N}Y_n(\tau+S)-Y_n(\tau+S-1)=0$$
	for all $\tau\in[0,1)$ and it follows
	$$\lim_{T\rightarrow\infty}T^{-|n|/2}\sum_{S=1}^{\lfloor T\rfloor}Y_n(\tau+S)-Y_n(\tau+S-1)=0$$
	for $\tau=T-\lfloor T\rfloor$. Since
	$$Y_n(T)=\sum_{S=1}^{\lfloor T\rfloor}Y_n(\tau+S)-Y_n(\tau+S-1)+Y_n(\tau),$$
	the claim in Lemma \ref{lem4} follows if we show that $T^{-|n|/2}Y_n(\tau)$ goes almost surely to zero. But by the Cauchy-Schwarz inequality
	$$Y_n(\tau)^2\leq \mathbb E_0\Big[I_n(\tau,W_{\tau})^2\Big]\mathbb E_0\bigg[\exp\Big\{2\beta\int_0^{\tau}\int_{\R^d}\phi(y-W_s)\dot B(s,\d y)\d s-\beta^2\tau V(0)\Big\}\bigg].$$
	The first factor on the right hand side is finite as $\tau$ is bounded by zero and one and all moments of a Gaussian random variable are finite. Again, because $\tau$ is bounded and as $\phi$ is a bounded function with compact support, 
	$$T^{-|n|/2}\bigg(\mathbb E_0\bigg[\exp\Big\{2\beta\int_0^{\tau}\int_{\R^d}\phi(y-W_s)\dot B(s,\d y)\d s-\beta^2\tau V(0)\Big\}\bigg]\bigg)^{1/2}\rightarrow 0$$
	almost surely, as $T\rightarrow\infty$.   
\end{proof}

We finally turn to the proof of Theorem \ref{key step}.
\begin{proof}[{\bf{Proof of Theorem \ref{key step}}}]
	By Lemma \ref{lem3}(a) we can rewrite the almost sure convergence statement in Lemma \ref{lem4} as
	\begin{align}
	\begin{split}
		\label{(2.2)}
		\lim_{T\rightarrow\infty}\mathbb E_0\bigg[&\sum_{i_1,..,i_d}A_n^{W_T}\Big(i_1,...,i_d,\frac{|n|-i_1-...-i_d}{2}\Big)\bigg(\frac{W_T^{\ssup 1}}{\sqrt{T}}\bigg)^{i_1}\cdot\cdot\cdot\bigg(\frac{W_T^{\ssup d}}{\sqrt{T}}\bigg)^{i_d}\\
		&\times \exp\Big\{\beta \int_0^T\int_{\R^d}\phi(y-W_s)\dot B(s,\d y)\d s-\frac{\beta^2T}{2}V(0)\Big\}\bigg]=0
	\end{split}
	\end{align}
	where we implicitly assume $(|n|-i_1-...-i_d)/2\in\N_0$.
	
	We now consider the sum
	$$\sum_{i_1,...,i_d}A_n^X\Big(i_1,..,i_d,\frac{|n|-i_1-...-i_d}{2}\Big)X^{\ssup 1^{i_1}}\cdot\cdot\cdot X^{\ssup d^{i_d}}$$
	with the i.i.d. random variables $X,X^{\ssup 1},...,X^{\ssup d}$, where $X\sim$ Normal$(0,1)$. The random variable $X$ has mean zero so that the odd moments of $X$ are zero. Since the random variables are independent, this means if at least one $i_k$ is odd, the expectation of that summand is zero. 
	A summand with no odd exponents arises from a summand where $i_1,...,i_{k-1},i_{k+1},...,i_d$ are even and $i_k$ is odd in the $|n|-1$-th derivative of $\exp\big\{\sum_{i=1}^{d}\lambda_iX^{\ssup i}-\frac{1}{2}\sum_{i=1}^d\lambda_i^2\big\}$. Note that by the product rule
	\begin{align}
	\begin{split}
		\label{product rule}
		\frac{\partial}{\partial\lambda_j}& \, \bigg[\big(X^{\ssup j}-\lambda_j\big)^{2k+1}\exp\Big\{\sum_{i=1}^d\lambda_i X^{\ssup i}-\frac{1}{2}|\lambda|^2\Big\}\bigg] \, \Bigg|_{\lambda=0}\\
		&=\bigg[\big(X^{\ssup j}-\lambda_j\big)^{2k+2}-(2k+1)\big(X^{\ssup j}-\lambda_j\big)^{2k}\bigg]\exp\Big\{\sum_{i=1}^d\lambda_iX^{\ssup i}-\frac{1}{2}|\lambda|^2\Big\}\bigg|_{\lambda=0}.
		\end{split}
	\end{align}
	We denote by $E$ the expectation w.r.t the random variables $X,X^{\ssup 1},...,X^{\ssup d}$. As $X$ is Normal$(0,1)$ distributed, 
	$E[X^{2k+2}]=(2k+1)E[X^{2k}]$
	for all $k\in\N_0$ and so by \eqref{product rule} one can obtain that the expectation of the summands with no odd exponents cancel each other. Both statements together (that for all exponents are even and that for at least one is odd) yield
	\begin{align}
		\label{(2.3)}
		E\bigg[\sum_{i_1,...,i_d}A_n^X\Big(i_1,..,i_d,\frac{|n|-i_1-...-i_d}{2}\Big)X^{\ssup 1^{i_1}}\cdot\cdot\cdot X^{\ssup d^{i_d}}\bigg]=0	
	\end{align}
	which shows that the sequence on the left hand side of \eqref{(2.2)} converges almost surely to the left hand side of \eqref{(2.3)} for every $n\in\N_0^d\setminus\{0\}$. By Lemma \ref{lem2} this holds true after dividing the sequence by $M_T$ and thus
	\begin{align}
	\begin{split}
	\label{2.2 and 2.3}
		\E^{\widehat{\mathbb{Q}}_{\beta,T}}&\bigg[\sum_{i_1,..,i_d}A_n^{W_T}\Big(i_1,...,i_d,\frac{|n|-i_1-...-i_d}{2}\Big)\bigg(\frac{W_{T}^{\ssup 1}}{\sqrt T}\bigg)^{i_1}\cdot\cdot\cdot\bigg(\frac{W_{T}^{\ssup d}}{\sqrt T}\bigg)^{i_d}\bigg]\\
		&\longrightarrow 	E\bigg[\sum_{i_1,...,i_d}A_n^X\Big(i_1,..,i_d,\frac{|n|-i_1-...-i_d}{2}\Big)X^{\ssup 1^{i_1}}\cdot\cdot\cdot X^{\ssup d^{i_d}}\bigg]
	\end{split}
	\end{align}
	almost surely.
	
	By Lemma $\ref{lem3}$(b), the coefficients $A_n^X\big(i_1,..,i_d,\frac{|n|-i_1-...-i_d}{2}\big)$ coming from
	$$\frac{\partial^{|n|}}{\partial\lambda_1^{n_1}\cdot\cdot\cdot\partial\lambda_d^{n_d}} \, \bigg[\exp\Big\{\sum_{i=1}^{d}\lambda_iX^{\ssup i}-\frac{1}{2}\sum_{i=1}^d\lambda_i^2\Big\}\bigg] \, \Bigg|_{\lambda=0}$$
	are equal to the coefficients $A_n^{W_T}\big(i_1,..,i_d,\frac{|n|-i_1-...-i_d}{2}\big)$.
	By induction, this statement yields for all $n\in\N_0^d\setminus\{0\}$
	\begin{align}
		\label{convergence to snd}
		\lim_{T\rightarrow \infty}\mathbb E^{\widehat{\mathbb{Q}}_{\beta,T}}\bigg[\bigg(\frac{W_{T}^{\ssup 1}}{\sqrt T}\bigg)^{n_1}\cdot\cdot\cdot\bigg(\frac{W_{T}^{\ssup d}}{\sqrt T}\bigg)^{n_d}\bigg]=E\bigg[X^{\ssup 1^{n_1}}\cdot\cdot\cdot X^{\ssup d^{n_d}}\bigg]\quad\text{a.s.}:
	\end{align}
	Indeed, for $|n|=1$ this is obvious. We check the induction step for $n\rightarrow \widetilde{n}$, where $\widetilde{n}=(n_1+1,n_2,...,n_d)$. The induction hypothesis is that (\ref{convergence to snd}) holds for all $m\in\N_0^d\setminus\{0\}$ that satisfy $m_1\leq n_1+1,...,m_d\leq n_d$ and at least one of these inequalities is strictly. By lemma \ref{lem3}(c) we know that the coefficient $A_{\widetilde{n}}^{W_{T}}(n_1+1,n_2,...,n_d,0)$ has value $1$. This means that by the induction hypothesis
	\begin{align}
	\begin{split}
	\label{by induction hypothesis}
		&\mathbb E^{\widehat{\mathbb{Q}}_{\beta,T}}\bigg[\sum_{i_1,..,i_d}A_{\widetilde{n}}^{W_{T}}\Big(i_1,...,i_d,\frac{|n|+1-i_1-...-i_d}{2}\Big)\bigg(\frac{W_{T}^{\ssup 1}}{\sqrt T}\bigg)^{i_1}\cdot\cdot\cdot\bigg(\frac{W_{T}^{\ssup 1}}{\sqrt T}\bigg)^{i_d}\bigg]\\
		&\qquad\qquad-\E^{\widehat{\mathbb{Q}}_{\beta,T}}\bigg[\bigg(\frac{W_{T}^{\ssup 1}}{\sqrt T}\bigg)^{n_1+1}\bigg(\frac{W_{T}^{\ssup 2}}{\sqrt T}\bigg)^{n_2}\cdot\cdot\cdot\bigg(\frac{W_{T}^{\ssup d}}{\sqrt T}\bigg)^{n_d}\bigg]\\
		&\longrightarrow E\bigg[\sum_{i_1,...,i_d}A_{\widetilde n}^X\Big(i_1,..,i_d,\frac{|n|+1-i_1-...-i_d}{2}\Big)X^{\ssup 1^{i_1}}\cdot\cdot\cdot X^{\ssup d^{i_d}}\bigg]\\
		&\qquad\qquad-E\Big[X^{\ssup 1^{n_1+1}}X^{\ssup 2^{n_2}}\cdot\cdot\cdot X^{\ssup d^{n_d}}\Big]
	\end{split}
	\end{align}
	almost surely as $T\to\infty$. If we set $n$ to be $\widetilde{n}$ in \eqref{2.2 and 2.3}, the induction step follows by combining \eqref{2.2 and 2.3} and \eqref{by induction hypothesis}. 
	Since (\ref{convergence to snd}) shows that all moments converge to the moments of the standard normal distribution, we have proved \eqref{claim} and thus Theorem \ref{key step}.
\end{proof}

\noindent{\bf{Acknowledgement:}} The authors would like to thank an anonymous referee for a very careful reading of the manuscript and many valuable comments.

\end{document}